\documentclass[a4paper, 11pt,twoside]{article}
\usepackage{color}
\usepackage{amssymb}
\usepackage{amsmath}
\usepackage{bbm}
\usepackage{mathrsfs}
\usepackage{tikz}
\usepackage{caption}
\usepackage{amsthm}
\usepackage{graphicx}
\usepackage{cite}
\usetikzlibrary{positioning}
\usepackage[margin=.95in]{geometry}

\newtheorem{theorem}{Theorem}

\newtheorem{lemma}{Lemma}

\newtheorem{example}{Example}

\newcommand{\E}{{\mathbb E}}
\newcommand{\Var}{\textrm{Var}}
\newcommand{\Cov}{\textrm{Cov}}

\usepackage{fancyhdr}
\begin{document}

\title{Self-averaging sequences which fail to converge}

\author{\small{\textbf{Eric Cator, Henk Don}}\\ \small{Faculty of
	Science, P.O. Box 9010, 6500 GL Nijmegen, The Netherlands;}\\ \small{\emph{email}: e.cator@science.ru.nl, h.don@math.ru.nl }}
\maketitle

\begin{abstract}
We consider self-averaging sequences in which each term is a weighted average over previous terms. For several sequences of this kind it is known that they do not converge to a limit. These sequences share the property that $n$th term is mainly based on terms around a fixed fraction of $n$. We give a probabilistic interpretation to such sequences and give weak conditions under which it is natural to expect non-convergence. Our methods are illustrated by application to the group Russian roulette problem.
\end{abstract}

\section{Introduction}

Suppose $n\geq 2$ people want to select a loser by flipping coins: all of them flip their coin and those that flip heads are winners. The others continue flipping until there is a single loser. This problem and generalizations of it have been extensively studied, see \cite{athreya,brands,bruss,eisenberg,grabner,kalpathy, kinney,li,louchard,prodinger,rade}. 
If at some stage all players flip heads before a loser is selected, we say the process fails. It is known that the probability of failure does not converge as $n$ increases. This sequence of probabilities is what we call a \emph{self-averaging sequence}. A similar problem is the shooting problem or group Russian roulette problem, as described by Winkler \cite{winkler}. Here players do not flip coins, but fire a gun on another player. Again one could ask for the probability on one survivor. Analysis of this problem is harder, since survival of an individual depends on survival of the other players. Recently van de Brug, Kager and Meester \cite{meester} rigorously showed that also here the sequence of probabilities does not converge and they gave bounds for the liminf and limsup. 

In this paper, we put such problems into a mild probabilistic framework and explain why this phenomenon of non-convergence is not surprising. The fact that in each round about the same fraction $\alpha$ of the players survives is the key ingredient to get oscillation instead of convergence. In the loser selection problem and the shooting problem the fluctuations around the fixed fraction are of order $\sqrt{n}$. We demonstrate that non-convergence of a self-averaging sequence is natural to expect under a much weaker condition: if the fluctuations are of order strictly less than $n$, the sequence should be expected not to converge. Our main theorem gives a way to bound the limit inferior and limit superior of a self-averaging sequence. Particular details of the problem do not really play a role in these bounds. 

To illustrate our general results, we applied our methods to the group Russian roulette problem. We obtain quite sharp upper and lower bounds on the liminf and limsup of the sequence of probabilities.
Non-convergence of the sequence follows immediately from these bounds.

\section{Problem formulation and running example}\label{sec:problem}

\subsection{General setting}

We will consider bounded sequences $p(n), n\geq 0$ which are defined as follows. The first term(s) are assumed to be given as starting values and then each next term is obtained by taking some weighted average over previous terms. This is a deterministic definition, but nevertheless we will adopt a natural probabilistic interpretation. The weighted average can be seen as the expectation of some random variable. So we will study a sequence $p(n)$ which is given for $0\leq n\leq n_0$ and satisfies
\begin{equation}\label{eq:def_p(n)}
p(n) = \mathbb{E}[p(Y(n))],\qquad n>n_0,
\end{equation}
where $Y(n)\in\left\{0,1,\ldots,n-1\right\}$ are random variables depending on $n$. For convenience we define $Y(n)=n$ for $n\leq n_0$. Furthermore, we assume that the expectation of $Y(n)$ is close to a fixed fraction of $n$ and that the variance around this fraction is of order $n$ as well. 

More precisely, we assume that there exist constants $0\leq \alpha<1$ and
$\beta,\gamma,\delta\geq 0$ such that for all $n>n_0$ the random
variable $Y(n)$ satisfies
\begin{equation}\label{eq:cond_p(n)}
|\mathbb{E}[Y(n)]-\alpha n|\leq \beta,\quad\textrm{and}\quad
\textrm{Var}(Y(n))\leq \gamma n+\delta.
\end{equation}
Sections \ref{sec:recursion} and \ref{sec:subsequences} deal with this general problem. Section \ref{sec:shootout} discusses a specific example: group Russian roulette, which is explained below. In Section \ref{sec:variance} we show that the condition on the variance can be weakened even further.   

\subsection{Running example: group Russian roulette}

To demonstrate our methods, we apply them to the group Russian roulette problem. Suppose in a group of $n$ people, each is armed with a gun. They all uniformly at random select one of the others to shoot at and they all shoot simultaneously. The survivors continue playing this
`game' until either one or zero survivors are left. The probability that in the end no survivor
is left is called $p(n)$. One characteristic of this problem is that in each round about the same fraction survives. Indeed, the probability for an individual to survive is $(1-\frac{1}{n-1})^{n-1}\approx \frac{1}{e}$, so the expected number surviving the first round is about $\frac{n}{e}$.

This problem was recently studied by Van de Brug, Kager and Meester \cite{meester}, who showed that $\lim_{n\rightarrow\infty}p(n)$ does not exist. In the current paper we will show that this phenomenon is a natural thing to expect under the quite general conditions of (\ref{eq:def_p(n)}) and (\ref{eq:cond_p(n)}).

\section{Analysis for fixed $n$: recursions in terms of expectation and variance}\label{sec:recursion}

Fix $n$ and define a sequence of random variables by
\begin{equation}\label{eq:X_k}
X_k = \left\{\begin{array}{lll}n&&\textrm{if}\ k=0\\Y(X_{k-1})&&\textrm{otherwise}\end{array}\right..
\end{equation}
In the setting of group Russian roulette, this means that the starting population has size $n$ and that $X_k$ is the number of survivors after $k$ rounds of shooting. 
When we condition on $X_k$, the number of survivors after the ($k+1$)st round is expected to be close to $X_k/e$ and the variance is of order $X_k$ as well (the precise constants will be derived in Section \ref{sec:shootout}). One might therefore expect that the number of survivors after $k$ rounds is about $X_0/e^k$. The next lemma shows that this guess is correct for the general case if (\ref{eq:def_p(n)}) and (\ref{eq:cond_p(n)}) hold. 

\begin{lemma}\label{lemma2} Suppose we have a constant $X_0>0$, random variables $X_1,X_2,\ldots$ and constants $0< \alpha<1$ and $\beta\geq 0$ such
	that for $k\geq 0$
	$$
	|\mathbb{E}[X_{k+1}\mid X_k] - \alpha X_k|\leq \beta\qquad a.s.
	$$
	Then for all $k\geq 0$
	
	$$
	|\mathbb{E}[X_k]- \alpha^k X_0| \leq \frac{\beta}{1-\alpha}.
	$$
\end{lemma}

\textbf{Proof.} We will prove by induction the following stronger
statement:
\begin{equation}\label{eq:stronger}
|\mathbb{E}[X_k] - \alpha^k X_0|\leq
\beta\sum_{i=0}^{k-1}\alpha^i\quad\quad\textrm{for}\ k\geq 0.
\end{equation} For $k=0$, the statement (\ref{eq:stronger}) is true. Now suppose
it holds for some $k\geq 0$. Then
\begin{eqnarray*}
	|\mathbb{E}[X_{k+1}]-\alpha^{k+1}X_0| &=&
	|\mathbb{E}[\mathbb{E}[X_{k+1}\mid X_k]]-\alpha^{k+1}X_0|\\
	&\leq& |\mathbb{E}[\alpha X_k+\beta]-\alpha^{k+1}X_0|\\
	&=& |\alpha\mathbb{E}[X_k]+\beta-\alpha^{k+1}X_0|\\
	&\leq& \left|\alpha\left(\alpha^k
	X_0+\beta\sum_{i=0}^{k-1}\alpha^i\right) +\beta
	-\alpha^{k+1}X_0\right|\\
	&=& \beta\sum_{i=0}^k\alpha^i 
\end{eqnarray*}\hfill $\Box$

To get further grip on the sequence $(X_k)_{k\geq 0}$, we also investigate the variance of the terms. It turns out that also the variance basically scales down with a factor $\alpha$ in each round.

\begin{lemma}\label{lemma3} Suppose we have a constant $X_0>0$, random variables $X_1,X_2,\ldots$ and constants $0< \alpha<1$ and $\beta,\gamma,\delta\geq 0$ such
	that for $k\geq 0$
	\begin{eqnarray}\label{eq:assump_lemmavar}
	&&|\mathbb{E}[X_{k+1}\mid X_k] - \alpha X_k|\leq
	\beta\qquad a.s.\label{eq:assump1_lemmavar}\\
	&&\textrm{\em Var}(X_{k+1}\mid X_k)\leq
	\gamma X_k+\delta\qquad a.s.\label{eq:assump2_lemmavar}
	\end{eqnarray}
	Then there exist constants $C$ and $D$, independent of $X_0$, such that for all $k\geq 0$
	\begin{equation}\label{eq:statement_lemma3}
	\textrm{\em Var}(X_{k}) \leq
	C\alpha^{k}X_0+D.
	\end{equation}
\end{lemma}

\textbf{Proof.} First we split the variance into two terms:
\begin{eqnarray}
\textrm{Var}(X_{k+1}) = \E[\Var(X_{k+1}\mid X_k)]+\Var(\E[X_{k+1}\mid X_k]).
\end{eqnarray}
For the first term, we use (\ref{eq:assump2_lemmavar}) and Lemma \ref{lemma2}:
\begin{eqnarray}
\E[\Var(X_{k+1}\mid X_k)] &\leq & \gamma\mathbb{E}[X_k]+\delta\nonumber\\
&\leq & \gamma\alpha^k X_0+\frac{\gamma\beta}{1-\alpha}+\delta.\label{eq:bound_firstterm}
\end{eqnarray}
For the second term, we use that (\ref{eq:assump1_lemmavar}) implies
\begin{equation}
\Var(\mathbb{E}[X_{k+1}\mid X_k] - \alpha X_k) \leq \E[(\mathbb{E}[X_{k+1}\mid X_k] - \alpha X_k)^2] \leq \beta^2.
\end{equation}
This gives
\begin{eqnarray}
&&\Var(\E[X_{k+1}\mid X_k])\ =\ \Var(\E[X_{k+1}\mid X_k]-\alpha X_k+\alpha X_k)\nonumber\\
&&\qquad\qquad =\ \Var(\mathbb{E}[X_{k+1}\mid X_k]-\alpha X_k)+\alpha^2\Var(X_k)+2\Cov(\E[X_{k+1}\mid X_k]-\alpha X_k,\alpha X_k)\nonumber\\
&&\qquad\qquad\leq\  \beta^2+\alpha^2\Var(X_k)+2\alpha\beta\sqrt{\Var(X_k)} = \alpha^2\left(\sqrt{\Var(X_k)}+\frac{\beta}{\alpha}\right)^2.
\end{eqnarray}
By elementary calculations, one can show that for all $x\geq 0$
\begin{equation}
\left(\sqrt{x}+c\right)^2 \leq \frac{K}{K-c^2}x+K,
\end{equation}
whenever $c$ is positive and $K>c^2$. This means that
\begin{eqnarray}\label{eq:bound_secondterm}
\Var(\E[X_{k+1}\mid X_k]) &\leq& \frac{\alpha^2K}{K-\frac{\beta^2}{\alpha^2}}\Var(X_k)+\alpha^2 K = \frac{\alpha^4K}{\alpha^2K-\beta^2}\Var(X_k)+\alpha^2K,
\end{eqnarray}
if $K > \beta^2/\alpha^2$. Now fix $k$ and assume that (\ref{eq:statement_lemma3}) holds true for this $k$. Using the bounds (\ref{eq:bound_firstterm}) and (\ref{eq:bound_secondterm}), we obtain
\begin{eqnarray}\label{eq:variance_k+1}
\textrm{Var}(X_{k+1}) &\leq & \left(\frac{\gamma}{\alpha}+C\frac{\alpha^3K}{\alpha^2K-\beta^2}\right)\alpha^{k+1}X_0+\left(\frac{\gamma\beta}{1-\alpha}+\delta+\alpha^2K+D\frac{\alpha^4K}{\alpha^2K-\beta^2}\right).
\end{eqnarray}
We want the constants in between brackets to be smaller than $C$ and $D$ respectively, i.e.
\begin{eqnarray}\label{eq:CD_inequalities}
\frac{\gamma}{\alpha}+C\frac{\alpha^3K}{\alpha^2K-\beta^2}\leq C,\qquad \frac{\gamma\beta}{1-\alpha}+\delta+\alpha^2K+D\frac{\alpha^4K}{\alpha^2K-\beta^2}\leq D
\end{eqnarray}
This can only be true if
\begin{equation}
\frac{\alpha^3K}{\alpha^2K-\beta^2} < 1\quad\textrm{and}\quad \frac{\alpha^4K}{\alpha^2K-\beta^2}<1.
\end{equation}
Since $\alpha<1$, the corresponding restriction on $K$ is $K>\frac{\beta^2}{\alpha^2-\alpha^3}$. If $K$ satisfies this inequality, then $K$ also exceeds $\beta^2/\alpha^2$ and (\ref{eq:CD_inequalities}) can be fulfilled by choosing $C$ and $D$ large enough. The minimal solutions are given by
\begin{eqnarray}\label{eq:constants}
C = \frac{\gamma\alpha^2K-\gamma\beta^2}{\alpha^3K-\alpha^4K-\alpha\beta^2},\qquad D = \frac{\left(\gamma\beta+\left(\delta+\alpha^2K\right)(1-\alpha)\right)\left(\alpha^2K-\beta^2\right)}{(1-\alpha)\left(\alpha^2K-\beta^2-\alpha^4K\right)}.
\end{eqnarray}
These solutions are both positive, so with this choice (\ref{eq:statement_lemma3}) holds for $k=0$. The inequalities (\ref{eq:variance_k+1}) and (\ref{eq:CD_inequalities}) complete a full inductive proof.\hfill $\Box$

\section{Bounds for subseqences of $p$}\label{sec:subsequences}

The previous section focussed on the random variables $X_k$ as defined in (\ref{eq:X_k}). Now we will use these results to study subsequences of $p$. For all $k\geq 1$ we obtain 
\begin{equation}
\E[p(X_k)] = \E[\E[p(Y(X_{k-1}))\mid X_{k-1}]] = \E[p(X_{k-1})],
\end{equation}
and hence for all $k\geq 0$
\begin{equation}\label{eq:expectation_p(X_k)}
\E[p(X_k)] = \E[p(X_0)]= p(X_0) = p(n).
\end{equation}
Since $X_k$ is expected to be close to $\alpha^kn$, we might hope that $p([\alpha^kn])$ is close to $p(n)$ if $p$ is smooth enough (where $[x] := \textrm{round}(x)$). We are interested in the limiting behavior if $n$ increases, so we will blow up $X_0$ by powers of $\alpha^{-1}$ and investigate the subsequence that emerges. Our main theorem is the following:

\begin{theorem}\label{theorem:main} Let $(p(n))_{n\in\mathbb{N}}$ be a sequence satisfying (\ref{eq:def_p(n)}) and (\ref{eq:cond_p(n)}). Choose $x\in\mathbb{R},x>0$ arbitrary and let $N_i = [\alpha^{-i}x]$. Then there exist intervals $I_0\subset I_1\subset I_2 \ldots$, centered at $x$ and having length of order $\sqrt{x}$, and a positive decreasing sequence $(q_k)_{k\in\mathbb{N}}$ for which $\sum_{k=0}^\infty q_k=1$ such that 
\begin{equation}\label{eq:main_theorem}
\sum_{k=0}^\infty q_k \min_{n\in \mathbb{N}\cap I_k}p(n) \leq p(N_i) \leq  \sum_{k=0}^\infty q_k \max_{n\in \mathbb{N}\cap I_k}p(n),
\end{equation}
for all $i\geq 0$.
\end{theorem}

Informally speaking, the idea of this theorem is that the values of $p(n)$ close to $n=x$ can be used to bound a whole subsequence of $p$. The problem setting suggests that $p(n)$ is roughly $f(\log(n))$, where $f$ is some periodic function with period $\log(\alpha)$. The scale on which the ``periodic" fluctuations occur in $p$ grows with the same speed as $n$, and therefore faster than $\sqrt{n}$. So if $x$ is large, we might expect $p(n)$ to vary only a little bit around $p([x])$ if $|n-x|$ is of order $\sqrt{n}$. This would imply that the subsequence $p(N_i)$ stays close to $p([x])$. Taking $x$ in a local maximum of the sequence $p(n)$, we can use (\ref{eq:main_theorem}) to bound $\limsup_{n\rightarrow\infty}p(n)$ from below. Similarly, we will construct an upper bound for $\liminf_{n\rightarrow\infty}p(n)$.\\ 
\\
\begin{proof} Let $X_k$ be defined as before by $X_k = Y(X_{k-1})$ for $k\geq 1$. We will consider these random variables for $X_0 = N_i$, $i\geq 0$. Define
\begin{eqnarray*}
	Z_i = X_i\mid (X_0 = N_i),\quad \mu_i = \mathbb{E}[Z_i],\quad \sigma^2_i =
	\textrm{Var}(Z_i).
\end{eqnarray*}
Then for all $i\geq 1$, using Lemma \ref{lemma2} and Lemma \ref{lemma3}, and incorporating
the rounding error
\begin{eqnarray}
|\mu_i-x|&\leq& |\mu_i-\alpha^i N_i|+|\alpha^iN_i-x|\leq \frac{\beta}{1-\alpha}+\frac12\\
\sigma^2_i &\leq&
C\alpha^iN_i+D \leq
C(x+\frac{\alpha}{2})+D. \label{eq:upperbounds}
\end{eqnarray}
So all $Z_i$ have expectation close to $x$ and by (\ref{eq:expectation_p(X_k)}) we have $p(N_i) = \mathbb{E}[p(Z_i)]$.  

Our main tool to control the subsequence $(p(N_i))_{i=1}^\infty$
will be the following version of Chebyshev's inequality:
$$
\mathbb{P}(| X-\mu| \leq t) \geq 1-\frac{\sigma^2}{t^2},
$$
where $X$ is a random variable with expectation $\mu$ and variance
$\sigma^2$. The upper bound for all variances $\sigma_i^2$ as given in (\ref{eq:upperbounds})
will be denoted by $\tau^2$. Application of Chebyshev's inequality
leads to
\begin{eqnarray*}
	\mathbb{P}(|Z_i-x|\leq t) &\geq&\mathbb{P}\left(|Z_i-\mu_i|\leq
	t-\frac{\beta}{1-\alpha}-\frac12\right)\\
	&\geq& 1-\frac{\tau^2}{(t-\frac{\beta}{1-\alpha}-\frac12)^2}
\end{eqnarray*}
Choosing $t=\tau+\frac{\beta}{1-\alpha}+\frac12$ gives for all
$i,k\geq 1$
\begin{equation}\label{eq:chebyshev}
\mathbb{P}(|Z_i-x|\leq t+k) \geq 1-\frac{\tau^2}{(\tau+k)^2}.
\end{equation}
Now we are ready to bound $p(N_i) = \mathbb{E}[p(Z_i)]$ as
follows:
\begin{eqnarray}
	p(N_i) &=& \mathbb{E}[p(Z_i)\mid |Z_i-x|\leq
	t+1]\cdot\mathbb{P}(|Z_i-x|\leq t+1)\nonumber\\
	&&+\sum_{k=1}^\infty \mathbb{E}[p(Z_i)\mid t+k<|Z_i-x|\leq
	t+k+1]\cdot\mathbb{P}(t+k<|Z_i-x|\leq t+k+1)\nonumber\\
	&\geq& \left(\min_{|n-x|\leq
		t+1}p(n)\right)\cdot\mathbb{P}(|Z_i-x|\leq t+1)\nonumber\\
	&&+\sum_{k=1}^\infty \left(\min_{|n-x|\leq
		t+k+1}p(n)\right)\cdot\mathbb{P}(t+k<|Z_i-x|\leq t+k+1)\nonumber\\
	&=& \min_{|n-x|\leq t+1}p(n)\cdot\mathbb{P}(|Z_i-x|\leq
	t)+\sum_{k=0}^\infty \left(\min_{|n-x|\leq
		t+k+1}p(n)\right)\cdot\left(\frac{\tau^2}{(\tau+k)^2}-\frac{\tau^2}{(\tau+k+1)^2}\right)\nonumber\\
	&& +\sum_{k=0}^\infty \left(\min_{|n-x|\leq
		t+k+1}p(n)\right)\cdot\left(\mathbb{P}(t+k<|Z_i-x|\leq t+k+1)-\frac{\tau^2}{(\tau+k)^2}+\frac{\tau^2}{(\tau+k+1)^2}\right),\nonumber\\
		&&\label{eq:expression}
\end{eqnarray}
where the minima are taken over $\mathbb{N}$. Now observe that 
$$
0\leq \mathbb{P}(|Z_i-x|\leq
	t+k)-1+ \frac{\tau^2}{(\tau+k)^2}\leq \frac{\tau^2}{(\tau+k)^2}
	$$
by (\ref{eq:chebyshev}). Since the right hand side is summable, the last sum $S$ in (\ref{eq:expression}) can be written as
\begin{eqnarray*}
	S &=& \sum_{k=0}^\infty \left(\min_{|n-x|\leq
		t+k+1}p(n)\right)\cdot\left(\mathbb{P}(|Z_i-x|\leq
	t+k+1)-1+\frac{\tau^2}{(\tau+k+1)^2}\right)\\
	&& -\sum_{k=0}^\infty \left(\min_{|n-x|\leq
		t+k+1}p(n)\right)\cdot\left(\mathbb{P}(|Z_i-x|\leq
	t+k)-1+\frac{\tau^2}{(\tau+k)^2}\right)\\
	&=& -\min_{|n-x|\leq t+1}p(n)\cdot\mathbb{P}(|Z_i-x|\leq t) +\\
	&& \sum_{k=1}^\infty\left(\min_{|n-x|\leq t+k}p(n)-\min_{|n-x|\leq
		t+k+1}p(n)\right)\cdot\left(\mathbb{P}(|Z_i-x|\leq
	t+k)-1+\frac{\tau^2}{(\tau+k)^2}\right)\\
	&\geq& -\min_{|n-x|\leq t+1}p(n)\cdot\mathbb{P}(|Z_i-x|\leq t).
\end{eqnarray*}
It follows that
\begin{eqnarray}\label{eq:lower_bound}
	p(N_i) &\geq & \sum_{k=0}^\infty \left(\min_{|n-x|\leq
		t+k+1}p(n)\right)\cdot\left(\frac{\tau^2}{(\tau+k)^2}-\frac{\tau^2}{(\tau+k+1)^2}\right).
\end{eqnarray}
Replacing the minimum by a maximum gives an upper bound for
$p(N_i)$. Now for $k\geq 0$ let
$$
I_k = [x-(t+k+1),x+(t+k+1)], \qquad q_k=\frac{\tau^2}{(\tau+k)^2}-\frac{\tau^2}{(\tau+k+1)^2}.
$$ 
Note that $\tau^2$ is of order $x$, so that $t$ is of order $\sqrt{x}$. Furthermore, $(q_k)_{k\in\mathbb{N}}$ is decreasing and $\sum q_k=1$, proving the result.\end{proof}

This theorem can be used to find upper and lower bounds for the liminf and limsup of the sequence $p(n)$:

\begin{lemma}\label{lemma:infsupbounds} For $x\in\mathbb{R^+}$, choose a lower bound $l(x)$ and upper bound $u(x)$ for $(p([\alpha^{-i}x]))_{i\geq 0}$ as in Theorem \ref{theorem:main}. Then for all $x$ 
\begin{equation}\label{eq:bounds1}
\liminf_{n\rightarrow\infty} p(n)\leq u(x),\qquad\qquad \limsup_{n\rightarrow\infty}p(n)\geq l(x).
\end{equation}
Furthermore, if $x_0>0$, then 
\begin{equation}\label{eq:bounds2}
\liminf_{n\rightarrow\infty} p(n) \geq \inf_{x\in[x_0,\alpha^{-1}x_0]} l(x),\qquad\qquad \limsup_{n\rightarrow\infty}p(n)\leq \sup_{x\in[x_0,\alpha^{-1}x_0]} u(x).
\end{equation}
\end{lemma}

\begin{proof}
For $x\in\mathbb{R}^+$ and $N_i = [\alpha^{-i}x]$, we have $l(x)\leq p(N_i)\leq u(x)$ for all $i$. This immediately gives the bounds in (\ref{eq:bounds1}). For the second statement, choose $n > \alpha^{-1}x_0$ arbitrary. Then there exists an integer $k\geq 1$ such that $\alpha^kn\in [x_0,\alpha^{-1}x_0]$. By definition $l(\alpha^kn)\leq p(n)\leq u(\alpha^kn)$. So for all $n > \alpha^{-1}x_0$, we have
$$
\inf_{x\in[x_0,\alpha^{-1}x_0]} l(x) \leq p(n) \leq \sup_{x\in[x_0,\alpha^{-1}x_0]} u(x),
$$  
which implies (\ref{eq:bounds2}).
\end{proof}

\section{Non-convergence in group Russian roulette}\label{sec:shootout}

In this section we will apply our methods to the group Russian roulette problem, as introduced in Section \ref{sec:problem}. We will see that it is quite straightforward to prove that the probability $p(n)$ to have no survivor in the end does not converge as the group size $n$ increases. 

We start by checking that the group Russian roulette problem indeed fits into our general framework (\ref{eq:def_p(n)}) and (\ref{eq:cond_p(n)}). The starting values are $p(0) = 1$ and $p(1) = 0$. Let $Y(n)$ be the number of survivors in a group of $n$ people after one round of shooting (with degenerate random variables $Y(0)=0$ and $Y(1)=1$). Then 
$$
p(n) = \sum_{k=0}^{n-1}\mathbb{P}(Y(n)=k)\cdot p(k) = \mathbb{E}[p(Y(n))].
$$
For $i=1,\ldots,n$, we define $I_i$ to be the indicator of the event
that individual $i$ survives the first round. We will calculate
the expectation $\mu_n$ and variance $\sigma_n^2$ of $Y(n)$ for $n\geq 2$.
\begin{eqnarray*}
\mu_n &=& \mathbb{E}\left[\sum_{i=1}^n I_i\right] = \sum_{i=1}^n \mathbb{P}(I_i = 1)\\
&=& n\left(1-\frac1{n-1}\right)^{n-1}.
\end{eqnarray*}
Next, we calculate the second moment of $Y(n)$.
\begin{eqnarray*}
\mathbb{E}\left[Y(n)^2\right] &=& \sum_{i=1}^n\mathbb{E}[I_i^2]+\sum_{i=1}^n\sum_{j=1,j\neq i}^n\mathbb{E}[I_iI_j]\\
&=& n\left(1-\frac1{n-1}\right)^{n-1}+(n^2-n)\mathbb{P}(I_1=I_2=1)\\
&=&
n\left(1-\frac1{n-1}\right)^{n-1}+(n^2-n)\left(1-\frac1{n-1}\right)^{2}\left(1-\frac2{n-1}\right)^{n-2}
\end{eqnarray*}
This gives
$$
\sigma_n^2 =
\mu_n-\mu_n^2+(n^2-n)\left(1-\frac1{n-1}\right)^{2}\left(1-\frac2{n-1}\right)^{n-2}.
$$
It can be shown that for all $n \geq 0$, 
$$
\left|\mu_n-\frac{n}{e}\right| \leq \frac{2}{e}\qquad
\textrm{and}\qquad
\sigma_n^2\leq \frac{e-2}{e^2}\cdot n+\frac{3-e}{2e^2},
$$
which means that we can choose the following constants in (\ref{eq:cond_p(n)}):
\begin{equation}
\alpha = \frac{1}{e},\qquad\beta = \frac{2}{e},\qquad\gamma = \frac{e-2}{e^2},
\qquad \delta = \frac{3-e}{2e^2}.
\end{equation}
Now we will compute the bounds of Lemma \ref{lemma:infsupbounds}. An expression for $l(x)$ is given in (\ref{eq:lower_bound}). To (approximately) calculate this function, we first need to find the constants $t,\tau$ and a range of values of $p(n)$. Values of $p(n)$ satisfy a recursive relation. 

\begin{figure}[h]
	\includegraphics[width = \textwidth]{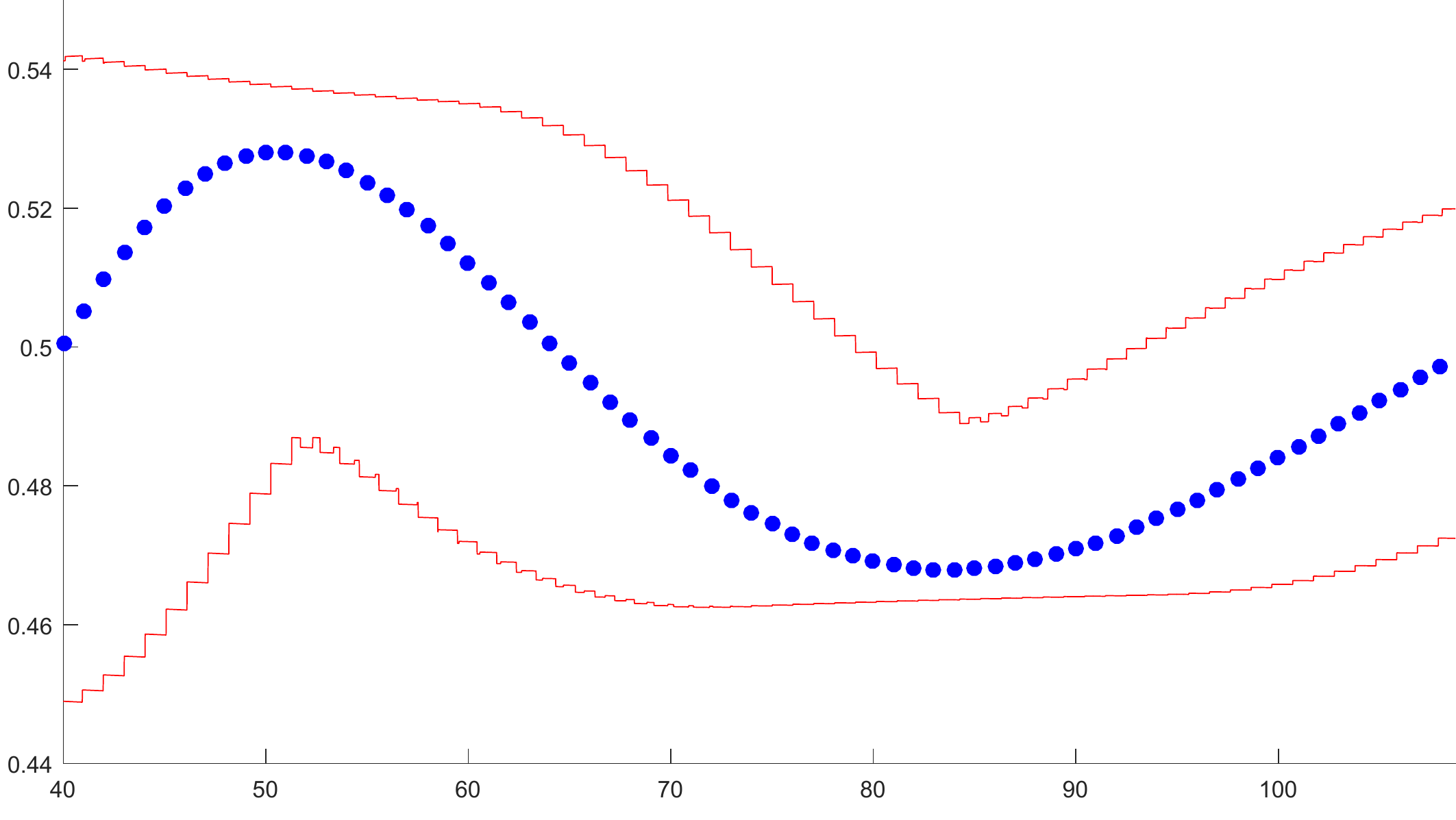}\caption{The values of $p(n)$ in one `period' (blue). For given $x$, the subsequence $p([\alpha^{-i}x]), i\geq 0$ stays between $l(x)$ and $u(x)$ (red).}\label{fig:shootout2}       
\end{figure}

Suppose we start with $n\geq 1$ people. Fix a subset of size
$1\leq k\leq n$ and denote the probability that exactly this
subset is killed in the first round by $q_{n,k}$. Then $q_{n,1} =
0$ and for $2\leq k \leq n$,
$$
q_{n,k} =
\left(\frac{k-1}{n-1}\right)^k\left(\frac{k}{n-1}\right)^{n-k}-\sum_{i=1}^{k-1}
\binom{k}{i}q_{n,i}.
$$
The recursion for $p(n)$ is the following:
\begin{equation}\label{eq:recursion_p}
p(n) = \sum_{k=0}^{n-2} \binom{n}{k}q_{n,n-k} p(k).
\end{equation}
We gratefully make use of the values for $p(n)$ as calculated by Van de Brug, Kager and Meester \cite{meester}. As a last ingredient, we calculate the constants $C$ and $D$ of Lemma \ref{lemma3} by equation (\ref{eq:constants}). Note that there is still the free parameter $K$ in the expressions for these constants, which should satisfy $K>\frac{\beta^2}{\alpha^2-\alpha^3} = \frac{4e}{e-1}\approx 6.33$. This constant can be used for fine-tuning of $C$ and $D$: increasing $K$ gives a smaller $C$ but a larger $D$. We  will choose $K=138$, because this appears to give the sharpest bounds in Lemma \ref{lemma:infsupbounds}. For given $x$, the terms in the sum in (\ref{eq:lower_bound}) can now be calculated explicitly, since $t$ and $\tau$ are determined by constants already known. A numerical lower bound for $(p(N_i))_{i\geq 0}$ is then obtained by performing this calculation for the first terms in the sum and bounding the tail by the uniform bounds $0\leq p(n)\leq 1$ for all $n$. An upper bound for $(p(N_i))_{i\geq 0}$ is calculated in an analogous way. 

As an illustration, we plotted $l(x)$ and $u(x)$ in Figure \ref{fig:shootout2} for $x\in[40,40e]$, which is one `period'. The sequence $p(n)$ itself is only defined on integers, but $l(x)$ and $u(x)$ are functions of a continuous variable. The discontinuities in these bounds are caused by a shifting window over which minima and maxima are taken in (\ref{eq:lower_bound}). 

To find bounds for the liminf and limsup of $p(n)$, we used the values of $p(n)$ as calculated by the recursion (\ref{eq:recursion_p}) in the range $0\leq n\leq 6000$. This results in the following theorem:

\begin{theorem}
	Let $p(n)$ be the probability that there are no survivors in the group Russian roulette problem with $n$ people. Then $\lim_{n\rightarrow\infty}p(n)$ does not exist. Moreover
	\begin{equation}
	0.4702 <\liminf_{n\rightarrow\infty}p(n) <0.4714\qquad\textrm{and}\qquad 0.5227 < \limsup_{n\rightarrow\infty} p(n) < 0.5237.
	\end{equation}
\end{theorem}

Figure \ref{fig:shootout} illustrates this result. The blue curve gives values of $p(n)$. The red curves are the bounds $l(x)$ and $u(x)$. For a fixed value of $x$, these curves give an interval containing all terms of the sequence $(p([\alpha^{-i}x]))_{i\geq 0}$. In particular, $\limsup_{n\rightarrow\infty}p(n)$ is bounded from below by the maximum of the lower red curve. Also each local maximum of the upper red curve is an upper bound for $\limsup_{n\rightarrow\infty}p(n)$, as is proved in Lemma \ref{lemma:infsupbounds}. Similar statements hold for $\liminf_{n\rightarrow\infty}p(n)$. The two horizontal lines indicate a band which will be left infinitely many times on both sides by values of $p(n)$. In \cite{meester}, it was shown that $\liminf_{n\rightarrow\infty}p(n) \leq 0.477487$ and $\limsup_{n\rightarrow\infty}p(n) \geq 0.515383$. So our bounds are an improvement over the results in \cite{meester}, despite the fact that our method does not rely on particular details of the group Russian roulette problem.

\begin{figure}[h]
	\includegraphics[width = \textwidth]{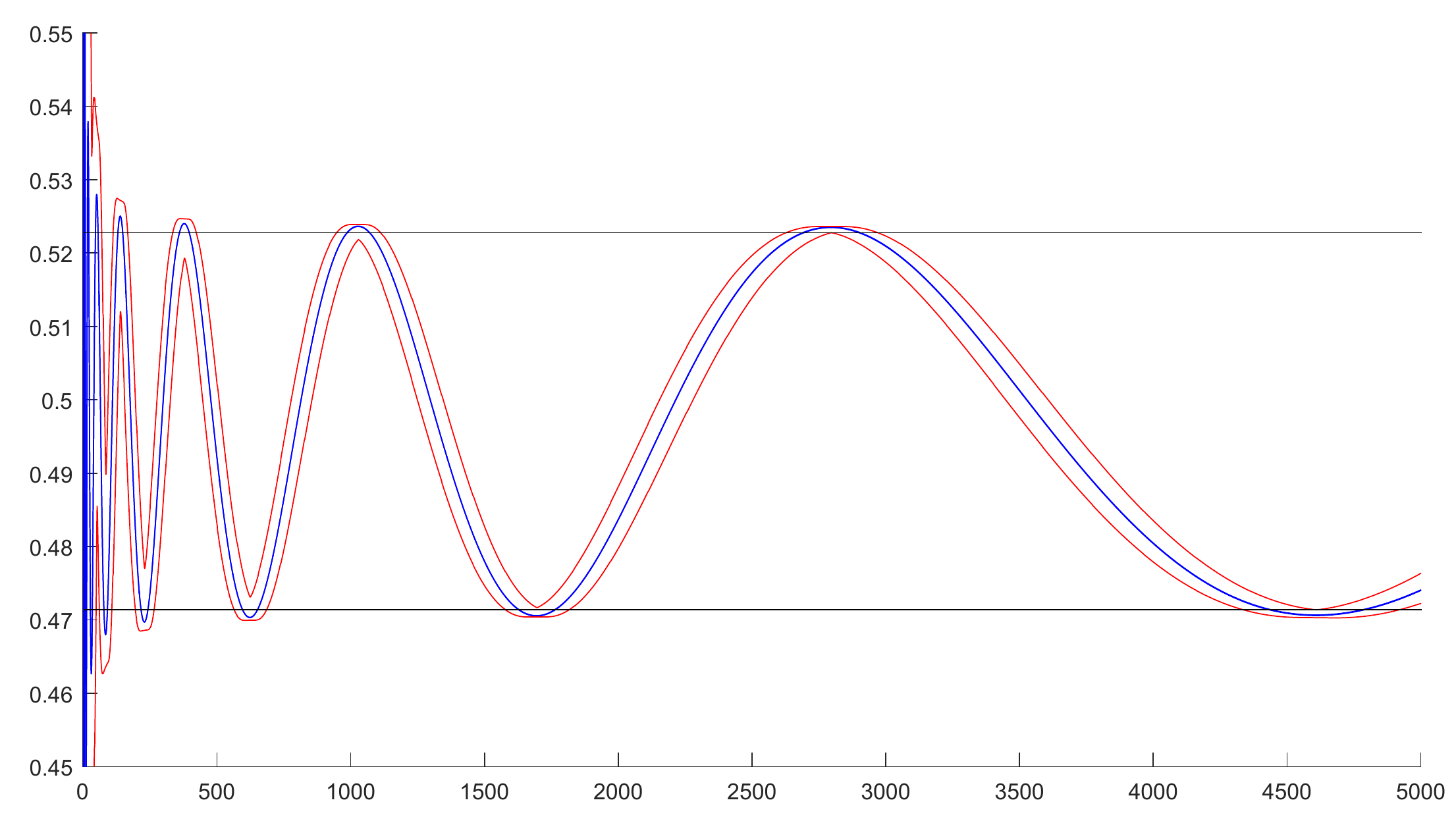}\caption{Plot of $p(n)$. Each vertical interval between the red curves contains an infinite subsequence. Horizontal lines indicate a minimal gap between $\limsup_{n\rightarrow\infty}p(n)$ and $\liminf_{n\rightarrow\infty}p(n)$.}\label{fig:shootout}       
\end{figure}

\section{Changing the order of the variance}\label{sec:variance}

In the setting of our problem, we assumed that the variance of $Y(n)$ is of order at most $n$, see (\ref{eq:cond_p(n)}). In fact, the phenomenon of non-convergence can even occur if the variance is of order $n^p$ with $p<2$ as the following generalization of Lemma \ref{lemma3} shows. That the ideas still work is not really surprising, since for $p<2$ the scale of the fluctuations in $Y(n)$ is still smaller than the scale of the periodic fluctuations in the sequence $(p(n))_{n\geq 0}$. If the power $p$ gets closer to $2$, the constants get worse, but the whole idea of subsequences which might have different limits essentially does not change. 

\begin{lemma}\label{lemma3_generalized} Suppose we have a constant $X_0>0$, random variables $X_1,X_2,\ldots$ and constants $0< \alpha<1$ and $\beta,\gamma,\delta\geq 0$ such
	that for $k\geq 0$
	\begin{eqnarray}\label{eq:assump_lemmavar_generalized}
	&&|\mathbb{E}[X_{k+1}\mid X_k] - \alpha X_k|\leq
	\beta\qquad a.s.\label{eq:assump1_lemmavar_generalized}\\
	&&\textrm{\em Var}(X_{k+1}\mid X_k)\leq
	\gamma X_k^p+\delta\qquad a.s.\label{eq:assump2_lemmavar_generalized}
	\end{eqnarray}
	for some $p<2$. Then there exist constants $C$ and $D$, independent of $X_0$, such that for all $k\geq 0$
	\begin{equation}\label{eq:statement_lemma3_generalized}
	\textrm{\em Var}(X_{k}) \leq
	C\alpha^{kp}X_0^p+D.
	\end{equation}
\end{lemma}

\textbf{Proof.} The scheme of the proof of Lemma \ref{lemma3} still works, it is only extended by some extra technical details. We start by bounding $\mathbb{E}[\textrm{Var}(X_{k+1}\mid X_k)]$, using Lemma \ref{lemma2} and Jensen's inequality for concave functions ($\frac{p}{2}<1$): 
\begin{eqnarray*}
	\mathbb{E}[\textrm{Var}(X_{k+1}\mid X_k)] &\leq&\gamma\E[X_k^p]+\delta \leq \gamma\left(\E[X_k^2]\right)^{\frac{p}{2}}+\delta = \gamma\left(\textrm{Var}(X_k)+\E[X_k]^2\right)^{\frac{p}{2}}+\delta\\
	&\leq& \gamma\left(\textrm{Var}(X_k)+\left(\alpha^kX_0+\frac{\beta}{1-\alpha}\right)^2\right)^{\frac{p}{2}}+\delta. 
\end{eqnarray*}
Assuming that the induction hypothesis (\ref{eq:statement_lemma3_generalized}) holds for some fixed $k$, we can further bound this as follows:
\begin{eqnarray}
\mathbb{E}[\textrm{Var}(X_{k+1}\mid X_k)] &\leq& \gamma\left(C\alpha^{kp}X_0^p+D+\left(\alpha^kX_0+\frac{\beta}{1-\alpha}\right)^2\right)^{\frac{p}{2}}+\delta \nonumber\\
&\leq& \gamma\left(\left(\sqrt{C}\alpha^kX_0\right)^2+\sqrt{C+D}^2+\left(\alpha^kX_0+\frac{\beta}{1-\alpha}\right)^2\right)^{\frac{p}{2}}+\delta \nonumber\\
&\leq& \gamma\left(\left(1+\sqrt{C}\right)\alpha^kX_0+\sqrt{C+D}+\frac{\beta}{1-\alpha}\right)^p+\delta \nonumber\\
&\leq& 2\gamma\left(1+\sqrt{C}\right)^p\alpha^{pk}X_0^p+2\gamma\left(\sqrt{C+D}+\frac{\beta}{1-\alpha}\right)^p+\delta. \label{eq_bound_first_term}
\end{eqnarray}
Here we have used that $(\alpha^kX_0)^p\leq (\alpha^kX_0)^2+1$ and $(x+y)^p\leq 2x^p+2y^p$ for $x,y\geq 0$. For the term $\textrm{Var}(\E[X_{k+1}\mid X_k])$, we obtain the bound of (\ref{eq:bound_secondterm}), after which the induction hypothesis (\ref{eq:statement_lemma3_generalized}) gives
\begin{eqnarray}\label{eq_bound_second_term}
\textrm{Var}(\E[X_{k+1}\mid X_k]) &\leq& \frac{\alpha^4K}{\alpha^2K-\beta^2}\Var(X_k)+\alpha^2K \leq \frac{\alpha^4K}{\alpha^2K-\beta^2}\left(C\alpha^{kp}X_0^p+D\right)+\alpha^2K,
\end{eqnarray}
whenever $K > \beta^2/\alpha^2$. Combining the bounds (\ref{eq_bound_first_term}) and (\ref{eq_bound_second_term}) leads to
\begin{eqnarray}\label{eq:variance_k+1_generalized}
\textrm{Var}(X_{k+1}) &\leq & \left(2\gamma\alpha^{-p}\left(1+\sqrt{C}\right)^p+\frac{\alpha^{4-p}K}{\alpha^2K-\beta^2}\cdot C\right)\alpha^{p(k+1)}X_0^p\nonumber\\
&&+\left(2\gamma\left(\sqrt{C+D}+\frac{\beta}{1-\alpha}\right)^p+\delta+\alpha^2K+\frac{\alpha^4K}{\alpha^2K-\beta^2}\cdot D\right).
\end{eqnarray}
To complete the proof, this needs to be smaller than $C\alpha^{p(k+1)}X_0^p+D$. For this to be true, we require that
\begin{equation}
\frac{\alpha^{4-p}K}{\alpha^2K-\beta^2} < 1\quad\textrm{and}\quad \frac{\alpha^4K}{\alpha^2K-\beta^2}<1,
\end{equation} 
which can be achieved by choosing $K>\frac{\beta^2}{\alpha^2-\alpha^{4-p}}$. Now since $p<2$ we can first choose $C$ and $D$ (both independent of $k$) large enough such that the upper bound in (\ref{eq:variance_k+1_generalized}) is indeed smaller than $C\alpha^{p(k+1)}X_0^p+D$. This finishes the inductive proof. \hfill $\Box$\\ 
\\
With this lemma a statement analogous to Theorem \ref{theorem:main} can be proved in the same way for the case when the variance of $Y(n)$ is of order $n^p, p<2$.

\section{Conclusions and remarks}

We have studied sequences $p(n)$ characterized by the property that each term is a weighted average over previous terms. In several examples in the literature, such sequences do not converge to a limit, which at first sight might be surprising. The main purpose of this paper is to demonstrate that it is natural to expect non-convergence if the largest weights in the average $p(n)$ are given to values $p(k)$ for which $k$ is close to a fixed fraction of $n$. It turns out that non-convergence is predictable or even inevitable under fairly weak conditions. The intuition is that fluctuations in $p$ happen on a large scale, and if the averages are taken on a smaller scale, they can not let the fluctuations vanish. Our methods are illustrated by proving non-convergence for the group Russian roulette problem. 

Another question one could ask is if $p(n)$ converges in the sense that there exists a periodic function $f:\mathbb{R}\rightarrow\mathbb{R}$ with period $1$ such that 
\begin{equation}\label{eq:periodic}
\lim_{x\rightarrow\infty} |p([\alpha^{-x}])-f(x)| = 0.
\end{equation}
As is shown in \cite{meester}, such a function exists in the case of group Russian roulette. However, the setting of (\ref{eq:def_p(n)}) and (\ref{eq:cond_p(n)}) is not sufficient to prove such convergence, as the following example demonstrates. This means that one would need stronger assumptions on the random variables $Y(n)$. We believe that for proving (\ref{eq:periodic}), a suitable requirement could be that the total variation distance between $Y(n)$ and $Y(n+1)$ goes to zero as $n$ increases. However, proving this goes beyond the scope of the current paper.
\begin{example}
{\em Let $p(0)=0$, $p(1)=1$ and define $Y(n),n\geq 2$ by
\begin{equation}
Y(n) = \left\{\begin{array}{ll}
2\cdot \textrm{Bin}\left(\frac12n,\frac12\right)&\qquad n\textrm{\ even}\\
2\cdot \textrm{Bin}\left(\frac12(n-1),\frac12\right)+1&\qquad n\textrm{\ odd}.
\end{array}\right.
\end{equation}
Then
\begin{equation}
\left|\mathbb{E}[Y(n)]-\frac12 n\right|\leq \frac12,\qquad\textrm{Var}(Y(n))\leq \frac12 n.
\end{equation}
Letting $p(n) = \mathbb{E}[p(Y(n))]$ gives a sequence fitting the framework of (\ref{eq:def_p(n)}) and (\ref{eq:cond_p(n)}). However, $Y(n)$ is even for $n$ even and odd for $n$ odd. This implies that  
\begin{equation}
p(n) = \left\{\begin{array}{ll}
0&\qquad n\textrm{\ even}\\
1&\qquad n\textrm{\ odd}.
\end{array}\right.
\end{equation}
In this case the function $p([x])$ is periodic, but $p([2^x])$ clearly is not periodic.\hfill $\blacksquare$
}\end{example}

As a final remark, we note that our methods also apply to a continous setting where $g:(0,\infty)\rightarrow\mathbb{R}$ is an absolutely bounded function and where $g(x)$ is given for $x\leq x_0$. In this case the recursion is of the form $g(x) = \mathbb{E}[g(N_x)],x>x_0$, where $N_x$ is a random variable supported on $(0,x)$.

\section*{Acknowledgements} We are obliged to Tim van de Brug, Wouter Kager and Ronald Meester for kindly providing the results of their calculations.


\begin{thebibliography}{3}
	
\bibitem{athreya}	J. S. Athreya and L. M. Fidkowski. \newblock {\em Number theory, balls in boxes, and the asymptotic
uniqueness of maximal discrete order statistics}, Integers (2000), A3, 5.

\bibitem{brands} J. J. A. M. Brands, F. W. Steutel, and R. J. G. Wilms, \newblock {\em On the number of maxima
in a discrete sample}, Statist. Probab. Lett. 20 (1994), no. 3, 209–-217. 

	\bibitem{meester} T. van de Brug, W. Kager, R. Meester. \newblock {\em The asymptotics of group Russian roulette}, (2015), arXiv:1507.03805.

\bibitem{bruss} F. T. Bruss and C. A. O'Cinneide.\newblock {\em On the maximum and its uniqueness for geometric
random samples}, J. Appl. Probab. 27 (1990), no. 3, 598–-610.
	
	
\bibitem{eisenberg} B. Eisenberg, G. Stengle, and G. Strang. \newblock {\em The asymptotic probability of a tie for first
place}, Ann. Appl. Probab. 3 (1993), no. 3, 731–-745. 

\bibitem{grabner} P. Grabner. \newblock {\em Searching for losers}, Random Structures and Algorithms 4 (1993), no. 1, 99–-110.

\bibitem{kalpathy} R. Kalpathy, H.M. Mahmoud, W. Rosenkrantz. \newblock{\em Survivors in leader election algorithms}, Statist. Probab. Lett. 83 (2013), no. 12, 2743–-2749. 

\bibitem{kinney} J. Kinney. \newblock {\em Tossing Coins Until All Are Heads}, Math. Mag. 51 (1978), no. 3, 184–-186.

\bibitem{li} S. Li and C. Pomerance. \newblock {\em Primitive roots: a survey}, S\={u}rikaisekikenky\={u}sho K\={o}ky\={u}roku
1274 (2002), 77–-87. New aspects of analytic number theory (Japanese) (Kyoto, 2001).

\bibitem{louchard} G. Louchard.\newblock{\em The asymmetric leader election algorithm: number of survivors near the end of the game}, Quaest. Math. 39 (2016), no. 1, 83–-101.

\bibitem{prodinger} H. Prodinger. \newblock {\em  How to select a loser}, Discrete Math. 120 (1993), no. 1-3, 149–-159.

\bibitem{rade} L. Rade, P. Griffin, and O. P. Lossers. \newblock {\em Problems and Solutions: Solutions: E3436},
Amer. Math. Monthly 101 (1994), no. 1, 78–-80. 



	
\bibitem{winkler}	P. Winkler. \newblock {\em Mathematical puzzles: a connoisseur’s collection}, A K Peters, Ltd., Natick,
MA, 2004. 

\end{thebibliography}
\end{document}